\documentclass[openany,reqno,a4paper,12pt]{amsart}

\usepackage{amsmath,amssymb,amsthm}

\usepackage{mathrsfs}

\usepackage{mathtools}
\usepackage{commath}

\usepackage{thmtools}
\usepackage{thm-restate}

\usepackage{cases}
\usepackage{enumitem}
\setlist[enumerate,1]{label=(\roman*)}

\usepackage[pdftex, pdfborderstyle={/S/U/W 0}]{hyperref}
\hypersetup{
    colorlinks=true,
    linkcolor=magenta,
    citecolor=cyan,
}

\usepackage{cleveref}

\usepackage{etoolbox}

\numberwithin{equation}{section}

\ifdefined\thmcolor
\declaretheoremstyle[
  shaded={bgcolor=\thmcolor}
]{plain}
\else
\fi

\ifdefined\defcolor
\declaretheoremstyle[
  headfont=\normalfont\bfseries,
  bodyfont=\normalfont,
  shaded={bgcolor=\defcolor}
]{noital}
\else
\declaretheoremstyle[
  headfont=\normalfont\bfseries,
  bodyfont=\normalfont,
]{noital}
\fi

\declaretheorem[style=plain,numberwithin=section,name=Theorem]{theorem}

\declaretheorem[style=plain,sibling=theorem,name=Lemma]{lemma}

\declaretheorem[style=plain,sibling=theorem,name=Claim]{claim}

\declaretheorem[style=plain,sibling=theorem,name=Question]{question}

\declaretheorem[style=plain,numbered=no,name=Theorem]{theorem-n}
\declaretheorem[style=plain,numbered=no,name=Proposition]{proposition-n}
\declaretheorem[style=plain,numbered=no,name=Lemma]{lemma-n}
\declaretheorem[style=plain,numbered=no,name=Corollary]{corollary-n}
\declaretheorem[style=plain,numbered=no,name=Conjecture]{conjecture-n}
\declaretheorem[style=plain,numbered=no,name=Claim]{claim-n}
\declaretheorem[style=plain,numbered=no,name=Fact]{fact-n}
\declaretheorem[style=plain,numbered=no,name=Open Problem]{openproblem-n}
\declaretheorem[style=plain,numbered=no,name=Question]{question-n}

\declaretheorem[style=noital,sibling=theorem,name=Definition]{definition}

\declaretheorem[style=noital,numbered=no,name=Remark]{remark-n}
\declaretheorem[style=noital,numbered=no,name=Definition]{definition-n}
\declaretheorem[style=noital,numbered=no,name=Construction]{construction-n}
\declaretheorem[style=noital,numbered=no,name=Example]{example-n}

\newcommand{\st}{\mathbin{\colon}}

\undef{\set}
\DeclarePairedDelimiter{\set}{\lbrace}{\rbrace}

\undef{\emptyset}
\newcommand{\emptyset}{\varnothing}

\newcommand{\from}{\colon}

\undef{\abs}
\DeclarePairedDelimiterX{\abs}[1]
  {\lvert}{\rvert}{\ifblank{#1}{\,\cdot\,}{#1}}

\undef{\norm}
\DeclarePairedDelimiterX{\norm}[1]
  {\lVert}{\rVert}{\ifblank{#1}{\,\cdot\,}{#1}}

\DeclarePairedDelimiterX{\inner}[2]
  {\langle}{\rangle}{\ifblank{#1}{\,\cdot\,}{#1},\ifblank{#2}{\,\cdot\,}{#2}}

\DeclareMathDelimiter{\given}
  {\mathbin}{symbols}{"6A}{largesymbols}{"0C}

\DeclareMathOperator{\Prob}{\mathbb{P}}
\DeclarePairedDelimiterXPP{\prob}[1]
  {\Prob}{\lparen}{\rparen}{}
  {\renewcommand{\given}{\nonscript\;\delimsize\vert\nonscript\;\mathopen{}}#1}

\DeclareMathOperator{\Expec}{\mathbb{E}}
\DeclarePairedDelimiterXPP{\expec}[1]
  {\Expec}{\lparen}{\rparen}{}
  {\renewcommand{\given}{\nonscript\;\delimsize\vert\nonscript\;\mathopen{}}#1}

\DeclareMathOperator{\Var}{Var}
\DeclarePairedDelimiterXPP{\var}[1]
  {\Var}{\lparen}{\rparen}{}
  {\renewcommand{\given}{\nonscript\;\delimsize\vert\nonscript\;\mathopen{}}#1}

\DeclareMathOperator{\Cov}{Cov}
\DeclarePairedDelimiterXPP{\cov}[2]
  {\Cov}{\lparen}{\rparen}{}{#1,#2}

\newcommand{\NN}{\mathbb{N}}

\newcommand{\cP}{\mathcal{P}}

\usepackage[numbers]{natbib}  

\usepackage{geometry}
\usepackage{comment}

\newcommand{\fim}{\mathbin{f_{k+1}}}
\newcommand{\fre}{\mathbin{f_{k}}}

\newcommand{\chrom}{\chi_\text{g}}
\newcommand{\blanks}{\chi_\text{gb}}
\newcommand{\connd}{\chi_\text{cg}}
\newcommand{\colcon}{\text{col}_\text{cg}}

\usepackage{tikz}
\usetikzlibrary {quotes} 

\begin{document}

\title{On monotonicity in Maker-Breaker graph colouring games}

\author{Lawrence Hollom}
\address{Department of Pure Mathematics and Mathematical Statistics (DPMMS), University of Cambridge, Wilberforce Road, Cambridge, CB3 0WA, United Kingdom}
\email{lh569@cam.ac.uk}



\begin{abstract}
    In the Maker-Breaker vertex colouring game, first publicised by Gardner in 1981, Maker and Breaker alternately colour vertices of a graph using a fixed palette, maintaining a proper colouring at all times.
    Maker aims to colour the whole graph, and Breaker aims to make some vertex impossible to colour.
    We are interested in the following question, first asked by Zhu in 1999: if Maker wins with $k$ colours available, must they also win with $k+1$?
    This question has remained open, attracting significant attention and being reposed for many similar games.
    While we cannot resolve this problem for the vertex colouring game, we can answer it in the affirmative for the game of arboricity, resolving a question of Bartnicki, Grytczuk, and Kierstead from 2008.

    We then consider how one might approach the question of monotonicity for the vertex colouring game, and work with a related game in which the vertices must be coloured in a prescribed order.
    We demonstrate that this `ordered vertex colouring game' does not have the above monotonicity property, and discuss the implications of this fact to the unordered game.

    Finally, we provide counterexamples to two open problems concerning a connected version of the graph colouring game.
\end{abstract}

\maketitle


\section{Introduction}
\label{sec:intro}

In this paper we are interested in Maker-Breaker graph colouring games, played on a graph $G$ between two players, Maker and Breaker, often referred to elsewhere as Alice and Bob.
The players take it in turns, starting with Maker, to assign one of a palette $X$ of colours to the vertices or edges of the graph, while obeying some restriction on the colouring (e.g. no adjacent vertices of the same colour, or no monochromatic cycles.)
Maker's goal is to colour the whole graph, and they win if this is accomplished, whereas Breaker wins if some vertex/edge becomes unplayable.

The best known example of such a game is the vertex colouring game and the associated game chromatic number.
This was first invented by Brams and then discussed by Gardner in his ``Mathematical Games'' column of Scientific American \cite{gardner1981scientific}. 
The game was then reinvented by Bodlaender \cite{bodlaender1991complexity}, who initiated the study of many properties of the game, which have since received significant attention.
The game is played on the vertices of a graph $G$, with the restriction that no two adjacent vertices may be given the same colour.
The \emph{game chromatic number} $\chrom(G)$ of $G$ is then the minimal number of colours $k$ for which Maker has a winning strategy.

The computational complexity of the vertex colouring game has also received attention.
Bodlaender \cite{bodlaender1991complexity} first asked about the complexity of the vertex colouring game in 1991, and it took almost thirty years for this problem to be resolved.
Indeed, in 2020 Costa, Pessoa, Sampaio, and Soares \cite{costa2020pspace} proved that the both the vertex colouring game and the related greedy colouring game are PSPACE-complete (see \cite{costa2020pspace} for more details).
This result has since been extended by Marcilon, Martins, and Sampaio \cite{marcilon2020hardness} to variants of the game where Breaker starts, or one or other of the players can pass their turn.

Despite the difficulty of computing $\chrom(G)$ in general, the game chromatic number of various classes of graphs has received significant attention.
In particular, the class of planar graphs and various subclasses thereof have been studied extensively \cite{faigle1991game, kierstead1994planar, zhu1999game, sekiguchi2014game, zhu2008refined, nakprasit2018game, sidorowicz2007game}, and there has also been work on random graphs in various regimes \cite{bohman2008game, keusch2014game, frieze2013game}.

Several Maker-Breaker graph colouring games besides the vertex colouring game have also been studied, for example the games corresponding to the chromatic index \cite{cai2001game}, majority colourings \cite{bosek2019majority}, and arboricity \cite{BARTNICKI20081388}.
The game chromatic number of matriods has also been studied \cite{lason2017coloring}, a generalisation of the game of arboricity that we will consider.
Another variation of the game which has been studied is the indicated chromatic number \cite{grzesik2012indicated}, wherein Maker selects a vertex and then Breaker selects the colour of that vertex.

In this paper, we are in particular interested with the following question concerning monotonicity in these games, and how the winner of the game changes with the number of colours available.

\begin{question}[{\cite[Question 1]{zhu1999game}}]
\label{question:vertex-monotone}
    If Maker wins the vertex colouring game with some number $k$ of colours, then does Maker win the vertex colouring game with $k+1$ colours?
\end{question}

This question was first posed by Zhu in 1999 \cite{zhu1999game} in the context of the vertex colouring game, where he noted that our intuition may suggest that the answer to the problem is `yes'.
Zhu also asked the following weaker version of the problem, although he noted that it may in fact not be at all easier to resolve than \Cref{question:vertex-monotone}.

\begin{question}
\label{question:weak-vertex-monotone}
    Is there a function $f$ such that $f(k)>k$, and if Maker wins the vertex colouring game with $k$ colours, then Maker also has a winning strategy with $l$ colours for any $l\geq f(k)$?
\end{question}

Despite being initially called `na\"{i}ve', the problem seems to be significantly deeper than it first appears, and has informed a number of further open problems concerning the other Maker-Breaker graph colouring games.
In particular, Bartnicki, Grytczuk, and Kierstead \cite{BARTNICKI20081388} studied the problem for the game of arboricity, played on the edges of a graph with the rule that there must be no monochromatic cycle.
They discuss the question of monotonicity, resolving an equivalent question to \Cref{question:weak-vertex-monotone} for this game, and \Cref{question:vertex-monotone} for the case of 2-degenerate graphs.

While we cannot resolve the previous two questions in their original form, we can resolve this question of monotonicity for the game of arboricity, and prove the following result.

\begin{theorem}
\label{thm:arboricity-monotone}
    If Maker wins the game of arboricity with $k$ colours, then Maker also wins with $k+1$ colours.
\end{theorem}

We prove \Cref{thm:arboricity-monotone} via a so-called `imagination strategy' for Breaker, and can convert a winning strategy for Breaker with $k+1$ colours to a winning strategy for Breaker with $k$ colours.
Such a method of proof was popularised by the seminal paper of Bre{\v{s}}ar, Klav{\v{z}}ar, and Rall \cite{brevsar2010domination}, who applied their techniques to the domination game on a graph.
Furthermore, we can convert a strategy from $k+1$ to $k$ colours so as to not change the location of any of Breaker's plays, only the colours played.

One might hope that a similar imagination strategy might resolve \Cref{question:vertex-monotone}, and that a strategy could be converted to a different number of colours without the need to change the location of any plays.
However, if this were the case, then we would also need the monotonicity property to hold for the vertex colouring game even when the order in which the vertices must be played is specified.
More precisely, given any total ordering $<$ on $V$, the vertex set of $G$, the \emph{ordered vertex colouring game} on $G$ subject to $<$ is the vertex colouring game wherein the order in which the vertices are coloured is prescribed by $<$ (but Maker and Breaker are still free to choose which colour they play.)
However, in \Cref{sec:ordered-vertex-colouring-game} we prove the following theorem, which shows that even the weaker \Cref{question:weak-vertex-monotone} is answered in the negative for the ordered vertex colouring game.

\begin{theorem}
\label{thm:ordered-non-monotone}
    For any integer $k$, and any $l>k$, there is a graph $G=(V,E)$ and ordering $<$ of $V$ such that Maker wins the ordered vertex colouring game with $k$ colours on $G$ subject to $<$, whereas Breaker wins the ordered vertex colouring game on $G$ subject to $<$ with $l$ colours.
\end{theorem}

Finally, we look at the connected vertex colouring game and connected game colouring number, wherein after the first move a vertex may only be played if it is adjacent to a previously-coloured vertex.
These were introduced by Charpentier, Hocquard, Sopena, and Zhu \cite{CHARPENTIER2020connected} in 2020, who proved some results and posed several problems about these games.
We follow their notation, writing $\connd(G)$ for the minimal number of colours for which Maker wins the connected vertex colouring game, and $\colcon(G)$ for the connected game colouring number of $G$.
Notably, they proved that for any $k$, there are graphs $G_k$ such that Maker wins the connected vertex colouring game with 2 colours, but Breaker wins with $k$ colours, and so this game does not have the monotonicity property corresponding to \Cref{question:weak-vertex-monotone}.

The connected vertex colouring game has received further study \cite{lima2023connected, bradshaw2023note}, but many questions remain open.
Here we provide counterexamples to two such questions.
We give a graph $G$ for which $\chrom(G)<\connd(G)$--that is, the connectivity helps Breaker--partially resolving a problem of Charpentier, Hocquard, Sopena, and Zhu \cite[Section 4, Problem 5]{CHARPENTIER2020connected}.
We also provide a graph $H$ and edge $e\in E(H)$ for which $\colcon(H-e) > \colcon(H)$, i.e. the connected game colouring number is not monotone in adding edges, unlike the usual game colouring number.
This resolves another problem of Charpentier, Hocquard, Sopena, and Zhu \cite[Section 4, Problem 4]{CHARPENTIER2020connected}.


\section{The game of arboricity}
\label{sec:game-of-arboricity}
The game of arboricity was introduced by Bartnicki, Grytczuk, and Kierstead \cite{BARTNICKI20081388} in 2008.
The game is played on the edges of a graph $G$, with neither player allowed to produce a monochromatic cycle. The set of edges given any particular colour form a forest, and so the minimal number of colours for which Maker wins is bounded below by the arboricity of $G$.
Bartnicki, Grytczuk, and Kierstead asked in \cite{BARTNICKI20081388} whether the winner of the arboricity game is monotone in the number of colours, and were able to give a positive answer in the case of 2-degenerate graphs.
The following theorem resolves this question.

\begin{theorem}
    For a fixed graph $G$, if Breaker wins the arboricity game with $k+1$ colours, then Breaker also wins on $G$ with $k$ colours.
\end{theorem}

\begin{proof}
    As we build Breaker's strategy, we will produce a colouring of $G$ with $k$ colours, represented by the partial function $\fre\from E\to [k]$.
    Breaker will use the following imagination strategy.
    Breaker keeps track of an imaginary version of the game, played with $k+1$ colours on the same graph $G$, and producing partial colouring $\fim\from E\to [k+1]$ (so $\fre$ and $\fim$ will be updated as the game progresses).
    Whenever Maker makes some play in the real game, updating $\fre$, Breaker copies Maker's move over to the imaginary game, updating $\fim$, maintaining the colour if possible. 
    If it is not possible, then Breaker concedes defeat and loses, which we will in due course prove never happens.
    Breaker then uses their winning strategy for $k+1$ colours to decide on some play; say that they set $\fim(e)=c$.
    This is converted to a move in the real game as follows.
    \begin{itemize}
        \item If $c\leq k$, and setting $\fre(e)=c$ is a valid move, then Breaker does this.
        \item If $\fre(e)$ cannot be set to any colour without creating a cycle, then $e$ is unplayable, so Breaker has won.
        \item Otherwise, Breaker sets $\fre(e)$ to an arbitrary legal colour.
    \end{itemize}

    First note that the set of edges played is identical between the imagined and real games.
    Therefore it suffices to prove that Maker's move can always be copied over to the imagined game while preserving the colour, and that if Breaker wins the imagined game, then they also win the real game.

    We will say that two vertices $u,v$ are \emph{in the same $c$-component} of colouring $f$ if they are in the same connected component of $G$ restricted to only edges given colour $c$ by the colouring $f$.
    Note that this induces an equivalence relation on the vertices for each colour $c$.
    We now state and prove a claim concerning these $c$-components.
    \begin{claim}
    \label{claim:components}
        If $1\leq c\leq k$, and at some stage of the game $u,v$ are in the same $c$-component of $\fim$, then they are in the same $c$-component of $\fre$.
    \end{claim}
    \begin{proof}
        We proceed by induction.
        Firstly, note that the claim is certainly true before any edges have been coloured, so assume now that the claim holds after some edges have been coloured in both games.
        We consider colouring a further edge $e=xy$, and show the claim still holds.
        
        If the same colour is played in both games, then the result follows immediately, so it suffices to consider only Breaker's turn.
        If Breaker sets $\fim(e) = k+1$, then there is nothing to prove, as the $c$-components of $\fim$ do not change for any $c\leq k$, and the $c$-components of $\fre$ do not shrink.
        
        Thus the only condition left to consider is the case where Breaker plays $\fim(e) = c\leq k$ in the imagined game, and the corresponding play in the real game is $\fre(e)=c' \neq c$.
        This can only occur if setting $\fre(e) = c$ was not a valid play, i.e. $x$ and $y$ are already in the same $c$-component of $\fre$.
        But then the only change to $c$-components of $\fim$ resulting from Breaker's play is to put $x$ and $y$ in the same $c$-component, and then the new $c$-component of $\fim$ of these points is contained in the $c$-component of $\fre$ of the points, so the claim holds.

        This completes our proof of the inductive step, and so the claim follows.
    \end{proof}

    Now, if a colour $c$ is played by Maker in the real game on an edge $e=xy$, then $x,y$ are not in the same $c$-component in the real game, and so by \Cref{claim:components} are not in the same $c$-component in the imagined game either, so Maker's move can be copied to the imagined game, as required.

    Finally, if Breaker wins in the imagined game, then there is an uncoloured edge $e=xy$ where $x$ and $y$ are in the same $c$-component in the imagined game for every $c\leq k+1$. Thus by \Cref{claim:components}, they are in the same $c$-component in the real game for every $c\leq k$, and so $e$ cannot be coloured in the real game, and so Breaker wins there too, as required.
\end{proof}


\section{The ordered vertex colouring game}
\label{sec:ordered-vertex-colouring-game}

As discussed in the introduction, it is natural to try and employ an imagination strategy to prove the monotonicity of the vertex colouring game, much like in the previous section.
One would consider a winning strategy for Breaker using $k+1$ colours, and from this construct a winning strategy for Breaker with $k$ colours.
The simplest method one could hope to use to achieve this would be to adjust the colour of Breaker's plays, but not the location. 
Indeed, as we know nothing about $G$, modifying a strategy by changing the location of plays seems very difficult to analyse.

If such a proof method were to work, then we would know that the monotonicity property was preserved even when the order in which the vertices must be played is prescribed.
However, we now prove \Cref{thm:ordered-non-monotone}, showing that there need not be monotonicity in such conditions.

We first give a precise definition of the ordered vertex colouring game.
\begin{definition}
\label{def:ordered-vertex-colouring-game}
    The \emph{ordered vertex colouring game} is played on a finite simple graph $G=(V,E)$ with some fixed linear ordering $<$ of $V$, with some fixed finite palette $X$ of colours.
    Then Maker and Breaker alternately colour vertices of $V$ with colours from $X$, colouring the vertices in the increasing order given by $<$.
    Both players may only give vertex $v$ colour $c$ if no neighbour of $v$ already has colour $c$.
    If at some point this is impossible to achieve, then Breaker wins. 
    Otherwise, the whole graph will eventually be coloured, in which case Maker wins.
\end{definition}

\begin{figure}[ht]
    \centering
    \begin{tikzpicture}[scale=2,darkstyle/.style={rectangle,draw,fill=gray!40,minimum size=10}]
        \path[-] 
            (3.5,3) edge (1,2)
            (3.5,3) edge (2,2)
            (3.5,3) edge (3,2)
            (3.5,3) edge (4,2)
            (3.5,3) edge (5,2)
            (3.5,3) edge (6,2)
            (1,2) edge (1,1)
            (2,2) edge (2,1)
            (3,2) edge (3,1)
            (4,2) edge (4,1)
            (4,2) edge (5,2)
            (5,2) edge (6,2)
            (1,2) edge [bend left=30] (3.5,3.5)
            (3.5,3.5) edge [bend left=30] (6,2)
            (2,1) edge [bend right=80] (6,2)
            (3,1) edge [bend right=60] (6,2);
         \node [darkstyle] (1) at (3.5,3) {1};
         \node [darkstyle] (2) at (1,2) {2};
         \node [darkstyle] (3) at (2,2) {3};
         \node [darkstyle] (4) at (2,1) {4};
         \node at (2.4,2) {\ldots};
         \node at (2.4,1) {\ldots};
         \node [darkstyle] (3r) at (3,2) {$2r+1$};
         \node [darkstyle] (4r) at (3,1) {$2r+2$};
         \node [darkstyle] (5) at (4,1) {$2r+3$};
         \node [darkstyle] (6) at (4,2) {$2r+4$};
         \node [darkstyle] (7) at (1,1) {$2r+5$};
         \node [darkstyle] (8) at (5,2) {$2r+6$};
         \node [darkstyle] (9) at (6,2) {$2r+7$};
    \end{tikzpicture}
    \caption{The graph $H_r$, which exhibits the non-monotonicity of the ordered vertex colouring game.}
    \label{fig:non-monotone-ordered-graph}
\end{figure}

We now state and prove two lemmas, which together immediately allow us to conclude \Cref{thm:ordered-non-monotone} in the case where $k=3$.
We can then make a slight modification to our constructed graph to deduce \Cref{thm:ordered-non-monotone} in full.
Throughout the rest of this section, $H_r$ will be the graph on $2r+7$ vertices shown in \Cref{fig:non-monotone-ordered-graph}, with its ordering induced from the labels of the vertices.

\begin{lemma}
\label{lem:ordered-maker-wins-3}
    Maker wins the ordered vertex colouring game on $H_r$ with 3 colours.
\end{lemma}

\begin{lemma}
\label{lem:ordered-breaker-wins-4}
    Breaker wins the ordered vertex colouring game on $H_r$ with $3+r$ colours.
\end{lemma}

Before proving \Cref{lem:ordered-maker-wins-3} and \Cref{lem:ordered-breaker-wins-4}, we first give a brief discussion of the intuition behind why this non-monotonicity occurs.

When there are $k\geq 3$ colours available, the only vertex which may ever become impossible to colour in $H_r$ is vertex $2r+7$, as all other vertices have at most two edges to vertices appearing earlier in the order.
As vertex 1 is coloured first, the only way in which Breaker could win is if vertex $2r+7$ becomes impossible to colour.

When there are $k=3$ colours, Maker may use their plays at vertices $3,5,\dotsc,2r+1$ and $2r+3$ to control what Breaker can do at vertices $4,6,\dotsc,2r+2,2r+4,$ and $2r+6$.
As there are few colours, Maker can exert significant control over Breaker this way.

However, when there are $k\geq 4$ colours, Maker has much less control over Breaker; for example, Maker cannot use their play at vertex $2r+3$ to control precisely what Breaker must play at vertex $2r+6$. 
Because of this extra freedom, Breaker can force vertex $2r+7$ to become unplayable, and thus win.

We now present the above arguments formally.

\begin{proof}[Proof of \Cref{lem:ordered-maker-wins-3}]
    Say the palette is $\set{1,2,3}$.
    We may wlog assume that Maker plays colour 1 at vertex 1, and Breaker then plays colour 2 at vertex 2.
    Already the only colour playable at vertex $2r+7$ is colour 3, so Maker will win if (and only if) they can prevent Breaker from playing colour 3 at any of vertices $4,6,\dotsc,2r+2$, and $2r+6$.

    Maker can achieve this by first playing colour 3 at vertex 3, as then Breaker cannot colour vertex 4 with colour 3.
    Maker then repeats this strategy, giving all vertices $5,\dotsc,2r+1$ colour 3.
    Maker then assigns colour 2 to vertex $2r+3$, as then Breaker must give vertex $2r+4$ colour 3, preventing them from colouring vertex $2r+6$ with colour 3.
    
    Thus colour 3 will remain available for vertex $2r+7$, and so Maker wins.
\end{proof}

\begin{proof}[Proof of \Cref{lem:ordered-breaker-wins-4}]
    Say that the palette is $\set{1,2,3,\dotsc,3+r}$.
    As in the previous proof, we may wlog assume that vertex 1 takes colour 1 and vertex 2 takes colour 2. 
    To win, Breaker must give vertices $4,6,\dotsc,2r+2$, and $2r+6$ colours $3,\dotsc,3+r$ in some order.

    Firstly, whatever Maker plays at each of the vertices $3,5,\dotsc,2r+1$, they can only rule out one colour from being played at each of $4,6,\dotsc,2r+2$.
    Therefore at each of these even-numbered vertices, Breaker may play some colour not yet adjacent to vertex $2r+7$, reducing the number of colours available at vertex $2r+7$ by one for each play they make.
    Thus once vertex $2r+2$ has been coloured, there will only be one colour playable at vertex $2r+7$.
    Call this one remaining colour $c$.
    
    However Maker colours vertex $2r+3$, Breaker will have at least two choices for colouring vertex $2r+4$, and so in particular may give vertex $2r+4$ some colour $d\neq c$.
    Maker's play at vertex $2r+5$ is irrelevant, and so Breaker may then give vertex $2r+6$ colour $c$, making vertex $2r+7$ unplayable, and winning.
\end{proof}

We now combine the above two lemmas to deduce \Cref{thm:ordered-non-monotone}.

\begin{proof}[Proof of \Cref{thm:ordered-non-monotone}]
    First note that \Cref{lem:ordered-maker-wins-3} and \Cref{lem:ordered-breaker-wins-4} show that the graph $H_{l-3}$ is exactly as we require when $k=3$, namely that Maker wins with 3 colours and Breaker wins with $l$.
    We let the vertex set of $H_{l-k}$ be $\set{v_1,\dotsc,v_n}$, where $n=2(l-k)+7$ and the $v_i$ is the vertex with label $i$ in \Cref{fig:non-monotone-ordered-graph}.
    If $k> 3$, then we define $G$ to be the graph on vertex set $\set{u_1,\dotsc,u_{2(k-3)},v_1,\dotsc,v_n}$, where the vertices $\set{v_1,\dotsc,v_n}$ induce a copy of $H_{l-k}$, and the remaining vertices are given edges as follows.
    For each $1\leq i\leq k-3$, vertex $u_{2i-1}$ is an isolated vertex, and $u_{2i}$ has edges to all $u_{2j}$ for all $1\leq j\leq k-3$ with $i\neq j$.
    We also add edges between $u_{2i}$ and $v_h$ for all $1\leq h\leq n$.
    The vertices of $G$ are ordered so that $u_1<u_2<\dots<u_{2(k-3)}<v_1<v_2<\dots<v_n$.

    As the game progresses, Maker's plays at the odd-indexed vertices in the set $U=\set{u_1,u_2,\dotsc,u_{2(k-3)}}$ are irrelevant, as those vertices are independent.
    Each time Breaker plays at an even-indexed vertex in $U$, they must use a new colour, as $\set{u_{2i}\st 1\leq i\leq k-3}$ induces a clique.
    Thus if $c$ colours were available at the start of the game, then by the time Maker plays at vertex $v_1$ there will be a palette of exactly $c-(k-3)$ colours remaining, and the same palette will be available at every remaining vertex.
    Thus the winner of the game will be the same as when playing on $H_{l-k}$ with $c-(k-3)$ colours.
    Thus we see that if $c=k$ then Maker wins by \Cref{lem:ordered-maker-wins-3}, whereas if $c=l$ then Breaker wins by \Cref{lem:ordered-breaker-wins-4}, as required.
\end{proof}

We also note that the graph $H_1$ has implications for the following question of Havet and Zhu \cite[Problem 5]{havet2013game}.

\begin{question}
\label{question:grundy-chromatic-comparison}
    Is it true that $\Gamma_g(G)\leq \chrom(G)$ for all graphs $G$?
\end{question}

Here, $\Gamma_g(G)$ is the \emph{game Grundy number} of $G$, introduced in \cite{havet2013game}, the minimal number of colours for which Maker wins the greedy colouring game.
In this game, Maker and Breaker alternately colour vertices of $G$, but now, instead of being able to pick the colour played, the palette $X$ of colours is linearly ordered and the players must play the minimal legal colour at every step.

As in the above discussion of monotonicity of the vertex colouring game, one can consider the \emph{ordered greedy colouring game} wherein an order of the vertices is prescribed.
This `game' is degenerate, as no choices remain for either player.
One might hope a proof of \Cref{question:grundy-chromatic-comparison} could be found by modifying a winning strategy from one game to also win the other, without changing the location of any plays. 
But this would imply that if Maker wins the ordered vertex colouring game with some (ordered) palette, then Maker would also win the ordered greedy colouring game on the same ordered graph and palette.
This does not hold however, as a simple check shows that Breaker wins the ordered greedy colouring game on the ordered graph $H_1$ with three colours.
Thus any resolution of \Cref{question:grundy-chromatic-comparison} requires a proof technique more subtle than that in \Cref{sec:game-of-arboricity}.


\section{The connected game chromatic number}
\label{sec:connected}

Charpentier, Hocquard, Sopena, and Zhu \cite{CHARPENTIER2020connected} asked five problems about the connected game chromatic number and connected game colouring number.
The first of their problems, concerning the connected game colouring number of outerplanar graphs, has been resolved by Bradshaw \cite{bradshaw2023note}, but the other four remained open.
We now prove \Cref{thm:chromatic-less-than-connected} and \Cref{thm:connected-not-edge-monot} two results, resolving the fourth and partially resolving the fifth problems of \cite{CHARPENTIER2020connected}.

We resolve these two problems by producing a counterexamples to each result.
Both of these examples were discovered via computer search, and the proofs that they have the claimed properties are via tedious case analysis, which is not presented here.
Instead, we give only brief intuitive hints as to why the properties are as we claim.

\begin{theorem}
\label{thm:chromatic-less-than-connected}
    There exists a graph $G$ for which $\chrom(G)<\connd(G)$.
\end{theorem}

The reason that connectivity helps Breaker in the graph shown in \Cref{fig:chromatic_less_than_connected} is roughly the following.
In the non-connected version of the game, Maker's only opening moves to win with 4 colours are to play at either vertex 1 or 2 (without loss of generality we may assume it is with colour 1).
If Breaker then replies with colour 2 at whichever of vertices 1 and 2 is available, then Maker's only winning move is to play at vertex 3, which is not possible in the connected version of the game.
Note further that once vertex 1 is played, all vertices other than 3 are playable, so we can easily translate between connected and non-connected versions of the game.

\begin{figure}[ht]
    \centering
    \begin{tikzpicture}[scale=2,darkstyle/.style={circle,draw,fill=gray!40,minimum size=10}]
        \path[-] 
            (2,1) edge (1,1)
            (2,1) edge (2,2)
            (2,1) edge (0.5,0.5)
            (2,1) edge (1,2)
            (2,1) edge (2.5,0.5)
            (1.5,0) edge (0.5,0.5)
            (1.5,0) edge (2.5,0.5)
            (1,1) edge (1,2)
            (2,2) edge (2.5,0.5)
            (2,2) edge (1,2)
            (2.5,0.5) edge (0.5,0.5)
            (1,2) edge (0.5,0.5)
            (0.5,0.5) edge [bend left=45] (0.5,2)
            (0.5,2) edge [bend left=45] (2,2);
         \node [darkstyle] (0) at (2,1) {1};
         \node [darkstyle] (1) at (1.5,0) {3};
         \node [darkstyle] (2) at (1,1) {};
         \node [darkstyle] (3) at (2,2) {2};
         \node [darkstyle] (4) at (2.5,0.5) {};
         \node [darkstyle] (5) at (1,2) {};
         \node [darkstyle] (6) at (0.5,0.5) {};
    \end{tikzpicture}
    \caption{A graph $G$ for which $\chrom(G)<\connd(G)$. Here, $\chrom(G)=4$ and $\connd(G)=5$.}
    \label{fig:chromatic_less_than_connected}
\end{figure}

\begin{theorem}
\label{thm:connected-not-edge-monot}
    There exists a graph $G$ with an edge $e\in E(G)$ for which $\colcon(G-e)>\colcon(G)$.
\end{theorem}

The reason the existence of edge $e$ in the graph shown in \Cref{fig:connected_not_edge_monot} helps Maker is roughly as follows.
In order for no vertex to be marked with more than two already-marked neighbours, Maker must start by marking the vertex labelled 1.
Then, if Breaker marks vertex 2, Maker's only winning reply is to mark vertex 3, which is only possible when the edge $e$ is present.

\begin{figure}[ht]
    \centering
    \begin{tikzpicture}[scale=2,darkstyle/.style={circle,draw,fill=gray!40,minimum size=10}]
        \path[-] 
            (2,3) edge (1,4)
            (2,3) edge (3,3)
            (3,3) edge (4,4)
            (3,3) edge (1,2)
            (4.5,3) edge (4,4)
            (4.5,3) edge (4,2)
            (2.5,1.5) edge (1,2)
            (2.5,1.5) edge [draw=red, dashed, "$e$"] (4,2)
            (4,4) edge (1,4)
            (4,4) edge (4,2)
            (1,4) edge (1,2)
            (1,2) edge (4,2);
         \node [darkstyle] (0) at (2,3) {};
         \node [darkstyle] (1) at (3,3) {};
         \node [darkstyle] (2) at (4.5,3) {2};
         \node [darkstyle] (3) at (2.5,1.5) {3};
         \node [darkstyle] (4) at (4,4) {};
         \node [darkstyle] (5) at (1,4) {};
         \node [darkstyle] (6) at (1,2) {};
         \node [darkstyle] (7) at (4,2) {1};
    \end{tikzpicture}
    \caption{A graph $G$ and edge $e$ (marked in dashed red) for which $\colcon(G-e)>\colcon(G)$. Indeed, $\colcon(G)=3$ and $\colcon(G-e)=4$.}
    \label{fig:connected_not_edge_monot}
\end{figure}


\section{Future work and open problems}
\label{sec:future}

There are many problems left open in the field of graph colouring games.
Of the problems which have been discussed in this paper, the most prominent is \Cref{question:vertex-monotone} and its equivalent versions for other games.
In light of the results presented in \Cref{sec:ordered-vertex-colouring-game}, it is unclear how to approach this problem, or indeed what the answer might be.
We also note that \Cref{question:grundy-chromatic-comparison} is still open.

As discussed in \Cref{sec:intro}, similar questions of monotonicity have been asked of several other graph properties, but very few solutions have been found.
The connected version of the graph colouring game of Charpentier, Hocquard, Sopena, and Zhu \cite{CHARPENTIER2020connected}, studied in \Cref{sec:connected}, does not have this property of monotonicity, and examples were given in \cite{CHARPENTIER2020connected} of, given an integer $k$, a graph where Maker wins with 2 colours, but Breaker wins with with any number of colours between 3 and $k$ (inclusive).
These examples also demonstrate that $\chrom(G)$ cannot be bounded from above by a function of $\connd(G)$.
Conversely, we have shown that $\connd(G)$ is not bounded above by $\chrom(G)$, but the following -- and perhaps more difficult -- question remains open.

\begin{question}
    Is there a function $f\from\NN\to\NN$ such that $\chrom(G)\leq f(\connd(G))$ for any connected graph $G$?
\end{question}

To the best of the author's knowledge, the result presented in \Cref{sec:game-of-arboricity} is the first nontrivial resolution in the affirmative to a problem of monotonicity of the form of \Cref{question:vertex-monotone}.
Despite the apparent difficulty, we believe that \Cref{question:hereditary} is interesting generalisation of these specific problems.
We now give the necessary definitions and then state this question.

Firstly, define a \emph{hereditary property of vertex (edge) coloured graphs}, to be a class of vertex (edge, respectively) coloured graphs, which is closed under taking subgraphs.
Two examples of such a property would be `is properly vertex-coloured', or `is edge-coloured without monochromatic cycles'.

For any hereditary property of vertex (edge) coloured graphs $\cP$, one may consider the Maker-Breaker graph colouring game with respect to $\cP$, wherein players colour vertices (edges, respectively) in turn, with the restriction that property $\cP$ must always hold.

\begin{question}
\label{question:hereditary}
    For which hereditary properties $\cP$ of vertex (edge) coloured graphs is it the case that if Maker wins the colouring game with respect to $\cP$ with $k$ colours, then Maker also wins with $k+1$ colours?
\end{question}

As there are a vast number of different hereditary properties of coloured graphs, an answer to \Cref{question:hereditary} in the form stated seems out of reach.
However, we believe that it would be interesting to investigate whether there are large classes of properties for which monotonicity could be proved or disproved, as this could potentially shed light on what the relevant characteristics of $\cP$ are for having/not having monotonicity.

One may of course also consider the weaker version of the above problem wherein we only ask that Maker wins with some number $f(k)$ of colours, but it is not clear whether this version of the problem is any more approachable than the one given above.

While we have made several observations about what structure proofs of such a statement for various specific $\cP$ would have to take, there still seems to be significant work required before a resolution to a question such as the above could be made.


\section{Declaration of Competing Interests}
The author declares that they have no known competing financial interests or personal relationships that could have appeared to influence the work reported in this paper.


\section{Acknowledgements}
\label{sec:acknowledgement}

The author is supported by the Trinity Internal Graduate Studentship of Trinity College, Cambridge.
The author would like to thank both their supervisor, Professor B\'{e}la Bollob\'{a}s, and the anonymous referee, for their thorough readings of the manuscript and many valuable comments.


\bibliographystyle{abbrvnat}  
\renewcommand{\bibname}{Bibliography}
\bibliography{main}


\end{document}